\documentclass[a4paper,flegn]{amsart}

\usepackage{amsthm,amsmath,amssymb}
\usepackage{latexsym}
\usepackage{bm}
\usepackage[all]{xy}
\usepackage{url}

\newtheorem{thm}{Theorem}[section]
\newtheorem{cor}[thm]{Corollary}
\newtheorem{lem}[thm]{Lemma}
\newtheorem{prop}[thm]{Proposition}
\theoremstyle{definition}
\newtheorem{defn}[thm]{Definition}
\theoremstyle{remark}
\newtheorem{rem}[thm]{Remark}

\newcommand{\Grz}{\mbox{\textbf{Grz}}}
\newcommand{\KM}{\textbf{KM}}
\newcommand{\Int}{\textbf{Int}}
\newcommand{\mHC}{\mbox{\textbf{mHC}}}
\newcommand{\GL}{\mbox{\textbf{GL}}}

\newcommand{\kfourGrz}{\textbf{K4.Grz}}
\newcommand{\Lan}{\mbox{$\mathcal{L}$}}

\newcommand{\fA}{\mbox{$\mathfrak{A}$}}

\newcommand{\fB}{\mbox{$\mathfrak{B}$}}

\newcommand{\al}{\alpha}
\newcommand{\be}{\beta}
\newcommand{\ga}{\gamma}

\newcommand{\ve}{\vee}

\newcommand{\we}{\wedge}

\newcommand{\ra}{\rightarrow}

\newcommand{\on}{\bm{1}}

\newcommand{\set}[2]{\{#1 ~|~#2\}}

\newcommand{\next}[1]{\text{NE}#1}

\newcommand{\km}{\textbf{km}} 

\begin{document}
	\date{}
	
	\title[On the Equipollence of $\Int$ and $\KM$]{On the Equipollence of the Calculi $\Int$ and $\KM$}
	\author[Alexei~Muravitsky]{Alexei~Y.~Muravitsky}
	\address{
		Louisiana Scholars' College\\
		Northwestern State University\\
		Natchitoches, LA 71495\\
		USA}
	\email{alexeim@nsula.edu}

\begin{abstract}
Following A. Kuznetsov's outline, we restore Kuznetsov's syntactic proof of the assertoric equipollence of the intuitionistic propositional calculus and the proof-intuitionistic calculus \textbf{KM} ({\em Kuznetsov's Theorem}). Then, we show that this property is true for a broad class of modal logics on an intuitionistic basis, which includes, e.g., the modalized Heyting calculus \textbf{mHC}. The last fact is one of two key properties necessary for the commutativity of a diagram involving the lattices of normal extensions of four well-known logics.  Also, we give an algebraic interpretation of the  assertoric equipollence for subsystems of $\KM$. 
\end{abstract}

\maketitle

\section{Introduction}\label{S:introduction}
In this paper we discuss the assertoric equipollence between the intuitionistic propositional calculus, denoted here by $\Int$, and the proof-intuitionistic calculus, $\KM$, which was introduced by A. Kuznetsov, though not in a form it is commonly known today and not under its today's name.\footnote{See our survey~\cite{mur14a} about how Kuznetsov and the author came to the definition of $\KM$. The equivalence (in a strong sense) between Kuznetsov's version for $\KM$ and the one which is defined below was shown in~\cite{km86}, pp. 82--82. The name ``KM'' is due to Leo Esakia~\cite{esa06}.} 

We remind the reader that $\KM$ is a modal system on the intuitionistic basis (see definition below) and closely related to the G\"{o}del-L\"{o}b provability logic $\GL$ so that $\KM$ was proved to be embedded into $\GL$; cf.~\cite{km80}, p. 224. Extending this embedding onto all normal extensions of $\KM$ made it possible to show that the lattices of normal extensions of $\KM$, $\next{\KM}$, and $\GL$, $\next{\GL}$, are isomorphic;~cf.~\cite{mur85,mur89}. This isomorphism along with Kuznetsov's Theorem (see below) led Kuznetsov and the author to the following commutative diagram:
\[
\xymatrix{
	\text{NE}\KM \ar@<1ex>[r]^{\tau} \ar@<1ex>[r];[]^{\rho}
	\ar[d]_{\lambda} &\text{NE}\GL\ar[d]^{\mu}\\
	\text{NE}\Int \ar@<1ex>[r]^\sigma  \ar@<1ex>[r];[]^{\sigma^{-1}} &\text{NE}\Grz	
}
\]
\begin{center}
	Diagram 1
\end{center}
where $\next{\Int}$ and $\next{\Grz}$ are the lattices of normal extensions of $\Int$ and the Grzegorczyk logic $\Grz$, respectively, $\tau$ and $\rho$ are lattice isomorphisms (and the inverses of one another), $\lambda$ and $\mu$ are join epimorphisms, and $\sigma$ is a well-known lattice isomorphism underlying the Blok-Esakia theorem; cf.~\cite{blok1976,esakia1976} and, also,~\cite{km86}.\footnote{Later on Diagram 1 was extended~\cite{mur17} by combining it with Diagram 2 (Section~\ref{S:conclusion}) which includes the lattices of the extensions of logics $\mHC$ and $\kfourGrz$; these logics were defined in~\cite{esa06}.}

The logic $\KM$, being a modal system on the intuitionistic basis, is not only a conservative extension of $\Int$, which can be obtained, for instance, from the finite model property for $\KM$ (see~\cite{mur81}), but also satisfies a stronger property: For any modality-free formulas $A$ and $B$,
\[
\KM+A\vdash B\Longleftrightarrow\Int+A\vdash B. \tag{\textit{Kuznetsov's Theorem}}
\]

Kuznetsov's Theorem makes it possible to show that $ \lambda $ is a semilattice epimorphism, and the whole diagram is commutative; see~\cite{km86} for detail.
Diagram 1, as well as Diagram 2 below, and their combination in~\cite {mur17} demonstrate a new view on the interaction of lattices of extensions of known logics.

The last equivalence was established by A.~Kuznetsov and stated as \textit{Theorem} in~\cite{kuz85b}. Because of the lack of space, the Theorem was preceded by a short (half-page) outline of its proof. Several attempts to prove Kuznetsov's Theorem algebraically have been unsuccessful until recently~\cite{mur08}.\footnote{The proof of Proposition~\ref{P:main} below turned to be useful in section 5 of~\cite{mur17b}.} In the present paper we prove a property of deducibility in $\KM$ (Proposition~\ref{P:main}), which expresses a somewhat stronger idea than the one which can be read in the Kuznetsov's outline and presented below as Corollary~\ref{C:pre-kuznetsov}. Then, Kuznetsov's Theorem is obtained as an easy consequence (Corollary~\ref{T:kuznestov}).

The paper is structured as follows. First,  in Section 2, we give our main definitions and obtain the deducibilities which will be used in the sequel. In Section 3, we prove our main result about $\KM$-deducibility, Proposition~\ref{P:main}, and derive some intermediate corollaries. In Section 4,
we define the notion of $\KM$-sublogic and show, how Kuznestov's Theorem can be extended to these systems. In the last section, we discuss, how the results of the preceding section can be applied to the modalized Heyting calculus $\mHC$ and other $\KM$-sublogics.

\section{Main definitions and some deducibilities}\label{S:basic-definitions}
The (propositional) language $\Lan$ is determined by the denumerable set $\lbrace p_{0},p_{1},\ldots\rbrace$ of  (propositional) variables and by the (logical) connectives: $ \wedge $ (conjunction), $ \vee $ (disjunction), $ \rightarrow $ (implication), $\neg$ (negation), and $ \square $ (modality). As usual, the parentheses, ``$($'' and ``$)$'', are used as punctuation marks.
The formulas (or $\Lan$-\textit{formulas}) are defined in a usual way with a usual agreement on the usage of parentheses. We define
\[
\on  ::=p_{0}\ra p_{0}
\]
and, as usual,
\[
\al\leftrightarrow\be::=(\al\ra\be)\we(\be\ra\al).
\]
Metavariables for $\Lan$-formulas will be denoted by $\al$, $\be$, $\ga$ (possibly with subscripts) while the letters $A$, $B$, $C$ will be used as
 metavariables for \textit{assertoric}, that is $\square$-\textit{free}, formulas. Thus the $\square$-free fragment of $\Lan$ will be used explicitly as a language, though we do not give it a name.
 A formula of the form $\square\ga$ is called a $\square$-\textit{formula}. Given a nonempty list of formulas, say $S=\langle \al_1,\al_2,\ldots,\al_n\rangle$, a $\square$-formula $\ga$ is called a \textit{maximal subformula} of $S$ if $\ga$ is a subformula of at least one formula of $S$ and for each $\al_i$,  where  $\ga$ is a subformula, $\ga$ does not occur in the scope of $\square$. For instance, any $\square$-formula which is maximal in length among all $\square$-subformulas
of $S$ is maximal in the above sense. In other words, a maximal formula is a maximal element in the partially ordered set of all $\square$-subformulas of $S$ arranged by the relation `$x$ is a subformula of $y$'.
Thus if $\square\alpha$ is a maximal subformula of $S$, it is not a subformula
of any $\square$-subformula of  $S$, except itself.
The set of all maximal subformulas of $S$ is denoted by $M(S)$. We note that $M(S)\subseteq\cup_{1\le i\le n} M(\langle\al_i\rangle)$, but not necessarily vise versa. By the \textit{rank} of $ S $ we mean the cardinality of $ M(S) $.\footnote{This definition of rank differs from the definition of Kuznetsov. We need this to reach some generalization (Proposition~\ref{P:main}) of Kuznetsov's original conclusion (Corollary~\ref{C:pre-kuznetsov}) for future reference.} As usual, by a \textit{substitution} we mean an endomorphism on the formula algebra of $\Lan$-formulas.

Given formulas $\al$, $\be$ and $\ga$, we denote the result of replacement of all occurrences of $\be$ in $\al$ with $\ga$ by
\[
\al[\be:\ga].
\]

As should be expected, the calculi $\Int$ and $\KM$ will be the key figures in our discussion.
The former is formulated in the assertoric fragment of $\Lan$ by the axioms $(\text{Ax}_0)$ below and the two rules of inference --- (simultaneous) \textit{substitution} and \textit{modus ponens}; the latter in full $\Lan$
 by the axioms $(\text{Ax}_0)$--$(\text{Ax}_3)$ below and the same rules of inference.

We will be dealing with several types of derivation, depending on the language and axioms employed. This is the full list of the axioms we deal with:
\[
\begin{array}{cl}
(\text{Ax}_{0})
&\mbox{axioms of intuitionistic propositional calculus},\\
&\mbox{e.g., corresponding to the schemata listed in~\cite{kle52}, {\S} 19};\\
(\text{Ax}_{1}) &p_{0}\ra\square p_{0};\\
(\text{Ax}_{2}) &(\square p_{0}\ra p_{0})\ra p_{0};\\
(\text{Ax}_{3}) &\square p_{0}\ra(p_{1}\ve{(p_{1}\ra p_{0}))}.
\end{array}
\]

The axioms $(\text{Ax}_{i})$ along with the inference rules, substitution and modus ponens, determine the following three consequence relations based on a corresponding notion of deducibility. Before turning to definitions, we want to make the following remark about the substitution rule. If we allow the use of any $\Lan$-formula in application of the substitution rule, we get one consequence relation, while if we restrict substitution to
$\square$-free formulas only, we get a different consequence relation. In the following definitions of types of deducibility, understood as a binary relation~$\vdash$, a usual notion of derivation is employed. 

We use the terms:
\begin{itemize}
\item $\KM$-\textit{deducibility} for $\KM +\al\vdash\be$, where all four $(\text{Ax}_{0})$--$(\text{Ax}_{3})$ can be used and substitution is allowed for all $\Lan$-formulas; 
\item $\Int^{\square}$-\textit{deducibility} for $\Int^{\square} +\al\vdash\be$, where only $(\text{Ax}_{0})$ can be used and substitution is allowed with no restrictions; 
\item $\Int$-\textit{deducibility} for $\Int + A\vdash B$, where  only axioms $(\text{Ax}_{0})$ can be used and substitution is restricted to the $\square$-free formulas.
\end{itemize}
In the $\KM$-, $\Int^{\square}$-
and $\Int$-deducibilities above, $\al$ and $A$ are called a \textit{premise} and $\be$ and $B$, respectively, a \textit{conclusion} of a derivation which supports a corresponding deducibility. Deducibilities without a premise are allowed and denoted by $\KM\vdash \be$, $\Int^{\square}\vdash\be$ and
$\Int\vdash B$, respectively.

We employ the letter $D$ (with or without a subscript) to denote a derivation.
Focusing on a derivation $D$, in order to indicate that $D$ supports $\KM+\al\vdash \be$ we will write
$D:\KM+\al\vdash \be$. This notation applies to all types of deducibility that we use. 

We remind the reader that two propositional calculi $\textbf{C}_1$ and $\textbf{C}_2$, where at least one of them is formulated in a modal language and both share their assertoric language, are called \textit{assertorically equipollent} if for any assertoric formulas $A$ and $B$, the following equivalence holds:
\[
\textbf{C}_1+A\vdash B \Longleftrightarrow \textbf{C}_2 +A\vdash B;
\]
compare with~\cite{km86}. 

Next we introduce derivations with special characteristics.

\begin{defn}[refined derivation]
A derivation is called refined if all substitutions, if any, apply only to the axioms occurring in the derivation or to the premise, if the derivation has a premise.
\end{defn}

\begin{rem}\label{R:refined}
	The derivations of all deducibility types  defined above can be made refined. To prove this, we can apply the technique of~\cite{sob74, lam79}. In the sequel, when we begin with a derivation, we assume that this derivation is refined.
\end{rem}
\begin{defn}[pure derivation, relation $\Vdash$]
A refined {\KM}- or {$\Int^{\square}$}-derivation $D$, that is when {$D:\KM+\al\vdash\be$} $($or, respectively, {$D:\Int^{\square}+\al\vdash\be$}$)$, is called pure if $M(D)\subseteq M(\langle\al,\be\rangle)$. We will use the notation
{$D:\KM+\al\Vdash\be$} $($or {$D:\Int^{\square}+\al\Vdash\be$}, respectively$)$ to indicate that $D$ is pure in these deducibilities.
We write simply  {$\KM+\al\Vdash\be$} $($or {$\Int^{\square}+\al\Vdash\be$}$)$
if there is a derivation $D$ such that {$D:\KM+\al\Vdash\be$} $($or, respectively, {$D:\Int^{\square}+\al\Vdash\be$}$)$.\footnote{Our definition of refined derivation is slightly more general than that of Kuznetsov.}
\end{defn}

As one can see, a pure derivation requires restrictions on applications of the substitution rule. For instance, in case of $\KM+A\Vdash B$, any pure derivation supporting this claim does not contain the modality $\square$.

It is quite obvious that
\begin{equation*}\label{E:int-equivalent}
\Int^{\square}+A\Vdash B\Longleftrightarrow
\Int+A\vdash B,
\end{equation*}
since $M(\langle A,B\rangle)=\emptyset$. For the same reason,
\begin{equation*}\label{E:km-equivalent}
\KM + A\Vdash B\Longleftrightarrow
\Int^{\square}+A\Vdash B.
\end{equation*}
This yields immediately
\begin{equation}\label{E:km-int-equivalent}
\KM + A\Vdash B\Longleftrightarrow
\Int +A\vdash B.
\end{equation}

\begin{prop}\label{P:one}
	For any formula $\alpha$, $\emph{\Int}^{\square}\vdash \alpha$ if and only if
	there is a $($$\square$-free$)$ formula $A$ such that $ \emph{\Int}\vdash A $ and $\alpha$ can be obtained from $A$ by substitution.
\end{prop}
\begin{proof}
The if-implication is obvious.
	The proof  of the only-if-implication is conducted by induction on the length $n$ of a given derivation
	$\Int^{\square}\vdash\alpha$. Indeed, if $n=1$ then $\al$ is $\square$-free and we can take $A=\al$.
	
	Now assume that $\al$ is derived
	in $\Int^{\square}$ by a derivation of length $n>1$. By virtue of Remark~\ref{R:refined}, this derivation is assumed to be refined. Therefore, either $\alpha$ is obtained by substitution from an $(\text{Ax}_{0})$-axiom $A$ or by modus ponens from $\be$ and $\be\ra\alpha$. In the first case, we arrive at the desired conclusion automatically. In the second case, there are formulas $B$ and $B\ra A$
	and a substitution $s$ such that $\Int\vdash B$ and $\Int\vdash B\ra A$ and also $\be\ra\alpha=s(B\ra A)$. Hence $\Int\vdash A$ and $\alpha=s(A)$.
\end{proof}

In the sequel, we will also need the following.
\begin{cor}\label{C:int-pure}
If $\emph{\Int}^{\square}\vdash \alpha$ then $\emph{\Int}^{\square}
\Vdash\alpha$. 
\end{cor}
\begin{proof}
Assume that $\Int^{\square}\vdash\alpha$.	According to Proposition~\ref{P:one}, $\Int\vdash A$, for some $A$, and $\alpha$ is obtained from $A$ by a substitution $s$, that is $s(A)=\al$. In view of Remark~\ref{R:refined}, the last deducibility can be made refined, that is $\Int\Vdash A$. Then, using the technique of~\cite{sob74} or that of~\cite{lam79}, the substitution $s$ can be ``pulled back'' to the involved axioms.
Let us denote the resulting derivation by $D$.
It should be clear that $M(D)=M(\langle\al\rangle)$. That is, ${\Int}^{\square}
\Vdash\alpha$.
\end{proof}

Now we consider some deducibilities which will be used in Section~\ref{S:Kuznetsov's-proof}.
We begin with the strong replacement property (Proposition~\ref{P:replacement}), which is a direct analogue of the replacement theorem for $\Int$-deducibilities, (see~\cite{kle52}, {\S} 26) and which is valid for $\Int^{\square}$.
\begin{prop}[the strong replacement property]\label{P:replacement}
Given formulas $A$, $B$ and $C$,
{\em
\[
\Int\vdash (A\leftrightarrow B)\rightarrow(C\leftrightarrow C[A:B]).
\]}
Analogously, given formulas $\al$, $\be$ and $\ga$, if $\al$ does not occur in any $\square$-subformula of $\ga$, then
\[
\emph{\Int}^{\square}\vdash(\al\leftrightarrow\be)\ra
(
\ga\leftrightarrow\ga[\al:\be]).\footnote{The strong replacement property without any conditions on $\al$ is true for $\KM$-deducibility; see~\cite{km86}.}
\]
\end{prop}


We proceed with the following two lemmas.
\begin{lem}\label{L:auxiliary}
Given formulas $A$ and $B$,
\[
\emph{\Int}\vdash
((A\ve(A\ra B))\ra B)\leftrightarrow B.
\]
\end{lem}
\begin{proof}
Indeed, we successively obtain:
\[
\begin{array}{l}

\Int\vdash
((A\ve(A\ra B))\ra B)
\leftrightarrow
((A\ra B)\we((A\ra B)\ra B)),\\
\Int\vdash
((A\ra B)\we((A\ra B)\ra B))
\leftrightarrow
((A\ra B)\we B),\\
\Int\vdash
((A\ra B)\we B)\leftrightarrow
 B.
\end{array}
\]
It remains to apply the property:
\[
\left[\Int\vdash C\leftrightarrow D~\text{and}~
\Int\vdash D\leftrightarrow E\right]\Longrightarrow
\Int\vdash C\leftrightarrow D;
\]
cf.~\cite{kle52}, {\S} 26.
\end{proof}

\begin{lem}\label{L:auxiliary-2}
Given formulas $A_{1},\ldots,A_{n}$ and $B$,
\begin{equation}\label{E:auxiliary}
\emph{\Int}\vdash
(\wedge_{1\le i\le n}(A_{i}\ve(A_{i}\ra B))\ra B)
\ra B.
\end{equation}
\end{lem}
\begin{proof}
We will reduce (\ref{E:auxiliary})
to a true statement by a number of invertible steps. On each step we use Proposition~\ref{P:replacement} (without mention) and sometimes Lemma~\ref{L:auxiliary}. Thus (\ref{E:auxiliary}) is equivalent to each of the following:
\[
\begin{array}{l}
\Int\vdash
(\wedge_{1\le i\le n-1}(A_{i}\ve(A_{i}\ra B))\ra ((A_{n}\ve(A_{n}\ra B))\ra B))
\ra B,\\
\Int\vdash
(\wedge_{1\le i\le n-1}(A_{i}\ve(A_{i}\ra B))\ra
 B)\ra B,~~~\![\mbox{Lemma~\ref{L:auxiliary}}]\\
\ldots\ldots\ldots\ldots\ldots\ldots\ldots\ldots
\ldots\ldots\ldots\ldots\ldots\ldots\ldots
~~
[\mbox{Lemma~\ref{L:auxiliary}}]\\
\Int\vdash((A_{1}\ve 
(A_{1}\ra B))\ra B)\ra B,~~~
[\mbox{Lemma~\ref{L:auxiliary}}]\\
\Int\vdash B\ra B.
\end{array}
\]
\end{proof}

\begin{defn}[rank of derivation, deducibility relation $\vdash_{m}$]\label{D:rank}
Given a refined derivation  $D:\KM+\alpha\vdash\beta$, the rank of $D$ is the cardinality of $M(D)$. If there is a refined  derivation $\KM+\alpha\vdash\beta$ of rank $m$, we write
$\KM+\alpha\vdash_{m}\beta$.
\end{defn}

We observe that 
$\KM+A\vdash_{0}B$ simply means that $\KM+A\Vdash B$.
Thus, in virtue of (\ref{E:km-int-equivalent}), we obtain:
\begin{equation}\label{E:km-int-2}
\KM+A\vdash_{0}B\Longleftrightarrow
\Int+A\vdash B.
\end{equation}

\section{Main results}\label{S:Kuznetsov's-proof}
In this section we prove our main result (Proposition~\ref{P:main}) and derive Kuznetsov's Theorem as its consequence (Corollary~\ref{T:kuznestov}). Also, we derive Kuznetsov's original key idea as Corollary~\ref{C:pre-kuznetsov}.

\begin{prop}\label{P:main}
	Let $D:\emph{\KM}+\al\vdash_{m}\be$ with $m>0$ and let $\square\ga\in M(D)$ so that $\square\ga$ is not a subformula of $\al$. Then there is a formula $\delta$ and a derivation $D_{1}: \emph{\KM}+\al\vdash_{l}\be[\square\ga:\delta] $ of rank $l<m$ and such that $\square\gamma\not\in M(D_{1})$ and, hence, $M(D_{1})\subset M(D)$.
\end{prop}
\begin{proof}
	Assume that
\begin{equation}\label{E:initial-gammas}
	D: \ga_1,\ga_2,\ldots,\ga_n.
\end{equation}
	Suppose all instances of
	$(\text{Ax}_{3})$ with $\square\ga$ as the antecedent used in $D$ are
\begin{equation}\label{E:axioms-three}
\square\ga\ra(\be_{1}\ve(\be_{1}\ra\ga)),\ldots,\square\ga\ra(\be_{k}\ve(\be_{k}\ra\ga)).
\end{equation}
Then we define
\begin{equation}\label{E:delta}
\delta::=
\begin{cases}
\begin{array}{cl}
\wedge_{1\le j\le k}(\be_{j}\ve(\be_{j}\ra\ga))[\square\ga:\on]
&\text{if (\ref{E:axioms-three}) is not empty}\\
\on &\text{if (\ref{E:axioms-three}) is  empty}.
\end{array}
\end{cases}
\end{equation}
Thus $\delta$ does not contain $\square\ga$. Then, we define
\[
\ga_{i}^{\ast}::=\ga_{i}[\square\ga:\delta], ~1\le i\le n,
\]
and consider the list
\begin{equation}\label{E:gammas-starred}
\ga_{1}^{\ast},\ga_{2}^{\ast},\ldots,\ga_{n}^{\ast}.
\end{equation}

Let us select any $\ga_i$ of (\ref{E:initial-gammas}) and examine the following cases. The goal of this examination is to show for each $\ga_{i}^{\ast}$  that either it is already a (refined) derivation
of type $\KM+\al\vdash_{l_{i}}\ga_{i}^{\ast}$ with $l_{i}<m$ or it can be extended to such a derivation. For each $i$, the resulting formula or sequence of formulas will be denoted below by $[\ga_{i}^{\ast}]$. As will be seen, $[\ga_{i}^{\ast}]$ does not contain $\square\ga$. Also, we will observe that $M([\ga_{i}^{\ast}])\subset M(D)$, for $\square\ga\notin M([\ga_{i}^{\ast}])$, and, therefore, the cardinality of each $M([\ga_{i}^{\ast}])$ is less than the cardinality of $M(D)$.
Then, concatenating  all derivations $[\ga_{i}^{\ast}]$, we get a derivation (denoted below by $D_1$) of $\ga_{n}^{\ast}=\be[\square\ga:\delta] $ of rank
$l$ which is less than $m$, since $l$ is the cardinality of $D_1$ and $M(D_1)\subset M(D)$.

Now we consider the forms in which each $\gamma_i$ may occur in $D$. We observe the following cases.

\begin{tabular}{cl}
(I) &Either $\ga_i$ is an instance of one of the axioms $(\text{Ax}_{0})$ or it is an instance of $\al$;\\
(II) &$\ga_i$ is an instance of $(\text{Ax}_{1})$ but is not $\ga\ra\square\ga$;\\
(II-$\ga$) &$\ga_i=\ga\ra\square\ga$;\\
(III) &$\ga_i$ is an instance of $(\text{Ax}_{2})$ but is not $(\square\ga\ra\ga)\ra\ga$;\\
(III-$\ga$) &$\ga_i=(\square\ga\ra\ga)\ra\ga$;\\
(IV) &$\ga_i$ is an instance of $(\text{Ax}_{3})$ but is not one of (\ref{E:axioms-three});\\
(IV-$\ga$) &$\ga_i$ is one of (\ref{E:axioms-three});\\
(V) &$\ga_i$ is obtained from $\ga_u$ and $\ga_v=\ga_{u}\ra\ga_{i}$ by modus ponens.\\ &(We remind the reader that the derivation \eqref{E:initial-gammas} is refined.)
\end{tabular}\\

In the cases (I), (II), (III) and (IV), $\ga_{i}^{\ast}$ has the form indicated in the corresponding case. Thus either $\ga_{i}^{\ast}$ is an instance of one of the axioms $(\text{Ax}_{0})$--$(\text{Ax}_{3})$ or is that of $\al$. We note that $\square\ga$ does not occur in $\ga_{i}^{\ast}$ and $\ga_{i}^{\ast}$ does not contain maximal formulas which would not be in $M(D)$.

In the case (II-$\ga$), $\ga_{i}^{\ast}=\ga\ra\delta$. If $\delta=\on$, then we have: $\Int\vdash p_{1}\ra\delta$ and hence $\Int^{\square}\vdash\ga\ra\delta$. Therefore, in virtue of Corollary~\ref{C:int-pure}, $\Int^{\square}\Vdash\gamma_{i}^{\ast}$. Next, if $\delta=\we_{1\le j\le k}(\be_{j}[\square\ga:\on]\ve(\be_{j}[\square\ga:\on]\ra\ga)$ (see~\eqref{E:delta}), we first notice that $\Int\vdash p_{0}\ra\we_{1\le j\le k}(p_{j}\ve(p_{j}\ra p_{0}))$ and hence, by virtue of Corollary~\ref{C:int-pure}, $\Int^{\square}\Vdash \ga\ra\delta$. Thus in both cases there is a derivation $\KM\vdash\ga_{i}^{\ast}$ of rank $l_{i}<m$, the set of maximal formulas of which is included in $M(D)$ but does not contain $\square\ga$. 

In the case (III-$\ga$), either $\ga_{i}^{\ast}=(\on\ra\ga)\ra\ga$ or 
$\ga_{i}^{\ast}=(\delta\ra\ga)\ra\ga$. If the former is the case, then we have:
$\Int\vdash(\on\ra p_{1})\ra p_{1}$ and hence (Corollary~\ref{C:int-pure}) $\Int^{\square}\Vdash \ga_{i}^{\ast}$.
Now, let the latter be the case. Since, by virtue of Lemma~\ref{L:auxiliary-2},
$\Int\vdash(\we_{1\le j\le k}(p_{j}\ve(p_{j}\ra p_{0}))\ra p_{0})\ra p_{0}$, we obtain, by Corollary~\ref{C:int-pure}, that
$\Int^{\square}\Vdash (\delta\ra\ga)\ra\ga$. Thus, as before, we conclude that there is a derivation $\KM\vdash\ga_{i}^{\ast}$ of rank $l_{i}<m$, the set of maximal formulas of which is included in $M(D)$ but does not contain $\square\ga$.

In the case (IV-$\ga$), 
$\ga_{i}^{\ast}=\delta\ra(\be_{j}[\square\ga:\delta]\ve(\be_{j}[\square\ga:\delta]\ra\ga))$, where $1\le j\le k$.
In view of the definition of $\delta$, (see~\eqref{E:delta}) $\Int^{\square}\vdash\delta\ra(\be_{j}[\square\ga:\on]\ve(\be_{j}[\square\ga:\on]\ra\ga))$. Also, since $\square\ga$ is maximal in $D$, $\on$ does not occur in any $\square$-subformula of $\be_{j}[\square\ga:\on]$; therefore, in virtue of Proposition~\ref{P:replacement},
\[
\Int^{\square}\vdash\delta\ra
((\be_{j}[\square\ga:\on]\ve(\be_{j}[\square\ga:\on]\ra\ga))\ra
(\be_{j}[\square\ga:\delta]\ve(\be_{j}[\square\ga:\delta]\ra\ga))).
\]
Therefore,
\[
\Int^{\square}\vdash\delta\ra
(\be_{j}[\square\ga:\delta]\ve(\be_{j}[\square\ga:\delta]\ra\ga)))
\]
and hence, according to Corollary~\ref{C:int-pure}, $\Int^{\square}\Vdash\ga_{i}^{\ast}$. And once again, we conclude that there is a derivation $\KM\vdash\ga_{i}^{\ast}$ of rank $l_{i}<m$, the set of maximal formulas of which is included in $M(D)$ but does not contain $\square\ga$.

Finally, in case (V), we observe that
$\ga_{i}^{\ast}$ can be obtained by modus ponens from $\ga_{u}^{\ast}$ and
$\ga_{v}^{\ast}$. By induction, we have: $\KM\vdash_{l_{u}}\gamma_{u}^{\ast}$ and 
$\KM\vdash_{l_{v}}\gamma_{v}^{\ast}$ with $\max\lbrace l_{u},l_{v}\rbrace<m$. Therefore, there is a derivation $\KM\vdash\gamma_{i}^{\ast}$ of rank less than $m$, which does not contain $\square\gamma$.

Now we define for each $i$, $1\le i\le n$, $[\ga_{i}^{\ast}]$ to be the $\KM$-derivation discussed in each case (I)--(V). As we have noted, each $M([\gamma_{i}^{\ast}])\subseteq M(D)$ and $\square\gamma\not\in M([\gamma_{i}^{\ast}])$.

Now we form
\[
D_{1}:[\ga_{1}^{\ast}],[\ga_{2}^{\ast}],\ldots,[\ga_{n}^{\ast}].
\]
Clearly, $D_{1}$ supports $\KM+\al\vdash
\be[\square\ga:\delta]$. Denoting its rank by $l$, we obtain that $l$ is less than or equal to the cardinality of 
$M([\gamma_{1}^{\ast}])\cup\ldots\cup M([\gamma_{1}^{\ast}])$ 
which is obviously less than $m$.
\end{proof}

We want to make the following observation. 
\begin{rem}\label{R:1} 
In the proof of Proposition~\ref{P:main}, in the construction of the derivations $[\gamma^{\ast}_{i}]$ only instances of $(\text{Ax}_{0})$ were employed.
\end{rem}

As an obvious consequence of Proposition~\ref{P:main}, we obtain the following.
\begin{cor}\label{C:pre-kuznetsov}
For any natural $m>0$, there is a nonnegative $l<m$ such that
\[
\emph{\KM}+A\vdash_{m}B\Longrightarrow
\emph{\KM}+A\vdash_{l}B.
\]	
\end{cor}

In~\cite{kuz85b} the statement of Corollary~\ref{C:pre-kuznetsov} was incorporated in the text and was a key step in Kuznetsov's argument supporting the following conclusive statement.
\begin{cor}[cf.~\cite{kuz85b}, Theorem]\label{T:kuznestov}
	The calculi \emph{\KM} and \emph{\Int} are assertorically equipollent, that is
	\[
	\emph{\KM}+A\vdash B\Longleftrightarrow
	\emph{\Int}+A\vdash B.
	\]
\end{cor}
\begin{proof}
	The $\Longleftarrow$ implication is obvious. To prove the $\Longrightarrow$ implication, assume that $\KM+A\vdash B$. Then, we  apply Corollary~\ref{C:pre-kuznetsov}, if necessary more than one time, to obtain $\KM+A\vdash_{0} B$. Then, we apply (\ref{E:km-int-2}).
\end{proof}

\section{$\KM$-Sublogics}\label{S:km-subsystems}

Introduced  as an intuitionistic counterpart of the provability logic $\GL$, (see~\cite{km86,mur14a}) logic $\KM$ has attracted attention of those researchers in the field who have been interested in modal systems on an intuitionistic basis, having ``provability smack.''  Among such systems we can mention $\textbf{K}^{\text{i}}$, $\textbf{K4}^{\text{i}}$, $\textbf{R}^{\text{i}}$, $\textbf{CB}^{\text{i}}$  and   $\textbf{SL}^{\text{i}}$ of~\cite{lit14}, and especially \textbf{mHC} of~\cite{esa06}, a close relative of $\KM$. These logics fall in the following group.

\begin{defn}[\KM-sublogic]\label{D:km-sublogic}
A set {\textbf{S}} of \Lan-formulas is called a {\KM}-sublogic if
\begin{center}
\begin{tabular}{cl}
$(i)$  &${\Int}^{\square}\vdash\alpha$ implies $\alpha\in{\textbf{S}}$;\\
$(ii)$ &$\alpha\in{\textbf{S}}$ implies ${\KM}\vdash\alpha$;\\
$(iii)$ &{\textbf{S}} is closed under substitution and modus ponens.
\end{tabular}
\end{center}

Deducibility  $\textbf{S}+\Gamma\vdash\alpha$ is understood in the sense that all the formulas valid in  \textbf{S}, as well as the formulas of $\Gamma$, can be used as axioms, and the rules of inference are substitution and modus ponens.
\end{defn}

The following observation follows quite obviously from Corollary~\ref{T:kuznestov}.
\begin{prop}\label{P:km-sublogic-equipolence}
Let \emph{\textbf{S}} be a \emph{\KM}-sublogic. Then
for any set $\Gamma\cup\lbrace A\rbrace$ of $\square$-free formulas,
the following conditions are equivalent$:$
\[
\begin{array}{cl}
(\text{\emph{a}}) &\emph{\Int}+\Gamma\vdash  A;\\
(\text{\emph{b}}) &\emph{\textbf{S}}+\Gamma\vdash A;\\
(\text{\emph{c}}) &\emph{\KM}+\Gamma\vdash A.
\end{array}
\]
In particular, \emph{\textbf{S}} and \emph{\Int} are assertorically equipollent.
\end{prop}

We recall (cf.~\cite{esa06}) that
\[
\mHC=\Int^{\square}+\square(p_{0}\rightarrow p_{1})\rightarrow
(\square p_{0}\rightarrow\square p_{1})+ (\text{Ax}_{1})
+ (\text{Ax}_{3}).
\]
Since
\[
\KM\vdash \square(p_{0}\rightarrow p_{1})\rightarrow
(\square p_{0}\rightarrow\square p_{1}),
\]
(cf.~\cite{km86}, p. 88) $\mHC$ is a \KM-sublogic. Then, in virtue of Proposition~\ref{P:km-sublogic-equipolence}, we straightforwardly obtain the following.
\begin{cor}\label{C:mHC-equipolence}
For any set $\Gamma\cup\lbrace A\rbrace$ of $\square$-free formulas,
\[
	\emph{\mHC}+\Gamma\vdash A\Longleftrightarrow
	\emph{\Int}+\Gamma\vdash A.
\]
\end{cor}
(Corollary~\ref{C:mHC-equipolence} answers in the affirmative a question in ~\cite{mur14a},  Problem 1.)

Now we turn to an algebraic interpretation of Proposition~\ref{P:km-sublogic-equipolence}.
\begin{defn}[\km-\textbf{S}-algebra, \mHC-algebra, {\KM}-algebra,
	\km-\textbf{S}-enrichable algebra, \mHC-enrichable algebra, {\KM}-enrichable algebra]
	\label{D:enrichable-algebras}
	Let \textbf{S} be a {\KM}-sublogic. An algebra $\fA=(\text{A},\wedge,\vee,\ra,\neg,\square)$, where
	$(\text{A},\wedge,\vee,\ra,\neg)$ is a Heyting algebra, is called a {\km-\textbf{S}}-algebra if for any formula $\alpha$,
	$\alpha\in{\textbf{S}}$ implies that
	$\fA$ validates $\alpha$. In particular, a {\km-\textbf{S}}-algebra is an {\mHC}-algebra if the unary operation $\square$ satisfies the following identities:
	\[
	\begin{array}{cl}
	(\emph{i}) &\square(x\wedge y)=\square x\wedge\square y,\\
	(\emph{ii}) &x\le\square x,\\
	(\emph{iii}) &\square x\le y\vee(y\ra x);
	\end{array}
	\]
	in addition, if the identity
	\[
	\begin{array}{cl}
	\!\!\!\!(\emph{iv}) &(\square x\ra x)\ra x=x
	\end{array}
	\]
	holds, the algebra is called a  {\KM}-algebra.
	An Heyting algebra {\fA} is
	{\km-\textbf{S}}-enrichable if a unary operation
	$\square$ can be defined in {\fA} in such a way that the resultant expansion $(\fA, \square)$ is a {\km-\textbf{S}}-algebra.
	In particular, {\fA} is
	 {\mHC}-enrichable if a unary operation $\square$ satisfies the conditions $(\emph{i}) - (\emph{iii})$; and is  {\KM}-enrichable if $(\emph{i})-(\emph{iv})$ are satisfied.\footnote{In~\cite{esa06} {\mHC}-algebras are called frontal Heyting algebras. In~\cite{kuz85b,km86} $\KM$-enrichable algebras are called $\Delta$-enrichable and in~\cite{mur08,mur14a} enrichable.}
\end{defn}

We easily observe the following.
\begin{lem}\label{L:km-enrichable-is-kms-enrichable}
Any \emph{\KM}-enrichable algebra is also \emph{\km-\textbf{S}}-enrichable.
\end{lem}
\begin{proof}
Let us fix a {\KM}-sublogic \textbf{S}. Assume that a Heyting algebra $\fA$ is {\KM}-enrichable. Then, according to Definition~\ref{D:enrichable-algebras}, there is an expansion $(\fA,\square)$ which validates $\KM$. In virtue of Definition~\ref{D:km-sublogic}-$(ii)$, $(\fA,\square)$ validates \textbf{S} as well.
\end{proof}

\begin{cor}\label{C:kms-enrichable-generate}
Let \emph{\textbf{S}} be a \emph{\KM}-sublogic. Any variety of Heyting algebras is generated by its \emph{\km-\textbf{S}}-enrichable algebras.
\end{cor}
\begin{proof}
Indeed, let a variety $\mathcal{V}$ of Heyting algebras be defined as an  equational class by a set $\Gamma$ of $\square$-free formulas. If the set $\Gamma$ is inconsistent, the variety $\mathcal{V}$ is trivial, that is it consists (up to isomorphism) of the one-element algebra, which is obviously {\km-\textbf{S}}-enrichable.

Next assume that $\Gamma$ is consistent. Then there is a formula $A$ such that $\Int+\Gamma\not\vdash A$. By virtue of Proposition~\ref{P:km-sublogic-equipolence}, $\textbf{S}+\Gamma
\not\vdash A$. Hence there is a \km-\textbf{S}-algebra $(\fA,\square)$ which separates $\Gamma$ from $A$. It is obvious that
$\fA$ is \km-\textbf{S}-enrichable.
\end{proof}
\begin{cor}
Any Heyting algebra is embeddable into
a \emph{\km-\textbf{S}}-enrichable algebra $($in particular, in \emph{\mHC}-enrichable algebra$)$ such that the latter generates the same variety of Heyting algebras as does the first algebra.
\end{cor}
\begin{proof}
It is known that any Heyting algebra $\fA$ is embedded into  such a \KM-enrichable algebra $\fB$ that $\fA$ and $\fB$ generate the same variety.\footnote{This property was announced as Corollary 2 of~\cite{kuz85b} and was derived from Kuznetsov's Theorem in~\cite{mur08}, Remark 3.} Let us fix any $\KM$-sublogic \textbf{S}.
By virtue of Lemma~\ref{L:km-enrichable-is-kms-enrichable}, $\fB$ is \km-\textbf{S}-enrichable.
\end{proof}

\section{Final remarks}\label{S:conclusion}
Let \textbf{S} be a $\KM$-sublogic. We define
\[
\lambda^{\ast}(\textbf{S})::=\set{A}{A\in\textbf{S}}
\]
and call $\lambda^{\ast}(\text{S})$ the \textit{assertoric fragment of} \textbf{S}. Let $\next{\textbf{S}}$ be the lattice of all normal extensions of \textbf{S}. Then, with the help of Proposition~\ref{P:km-sublogic-equipolence}, it can be proven that $\lambda^{\ast}$ is a join epimorphism of $\next{\textbf{S}}$ onto $\next{\Int}$. Moreover, if $\textbf{S}=\mHC$, then there are
a lattice isomorphism $\tau^{\ast}$ and a join epimorphism $\mu^{\ast}$ such that the following diagram is commutative: 
\[
\xymatrix{
	\text{NE}\mHC \ar@<1ex>[r]^{\tau^{\ast}} \ar@<1ex>[r];[]^{\tau^{\ast-1}}
	\ar[d]_{\lambda^{\ast}} &\text{NE}\kfourGrz\ar[d]^{\mu^{\ast}}\\
	\text{NE}\Int \ar@<1ex>[r]^\sigma  \ar@<1ex>[r];[]^{\sigma^{-1}} &\text{NE}\Grz	
}
\]
\begin{center}
	Diagram 2
\end{center}
where 
\[
\kfourGrz::=\textbf{K4} + \square(\square(p_0\rightarrow\square p_0)\rightarrow p_0)\rightarrow\square p_0
\]
and $\sigma$ is the same as in Diagram 1; cf.~\cite{mur15, mur17}. This inspires us to propose the following.\\

\noindent\textbf{Conjecture.} For any $\KM$-sublogic \textbf{S}, there is a normal extension \textbf{L} of \textbf{K}, which is included in $\GL$, and there are a lattice isomorphism $\tau^{\ast}$ and a join epimorphism $\mu^{\ast}$ such that the following diagram is commutative: 
\[
\xymatrix{
	\text{NE}\textbf{S} \ar@<1ex>[r]^{\tau^{\ast}} \ar@<1ex>[r];[]^{\tau^{\ast-1}}
	\ar[d]_{\lambda^{\ast}} &\text{NE}\textbf{L}\ar[d]^{\mu^{\ast}}\\
	\text{NE}\Int \ar@<1ex>[r]^\sigma  \ar@<1ex>[r];[]^{\sigma^{-1}} &\text{NE}\Grz	
}
\]
\begin{center}
	Diagram 3
\end{center}


\end{document}